\ifdefined\iscolt

\documentclass[anon,12pt]{colt2022} % Anonymized submission
\title[A gradient sampling method for general Lipschitz functions]{A gradient sampling method with complexity guarantees for  Lipschitz functions in high and low dimensions}
\usepackage{times}
\coltauthor{%
 \Name{Author Name1} \Email{abc@sample.com}\\
 \addr Address 1
 \AND
 \Name{Author Name2} \Email{xyz@sample.com}\\
 \addr Address 2%
}

\usepackage{mathrsfs}  
\usepackage{algorithm}
\newcommand{\RR}{\mathbb{R}}
\newcommand{\NN}{\mathbb{N}}

\newcommand{\dom}{{\mathrm{dom}}} % domain
\DeclareMathOperator*{\argmin}{arg\,min}

\usepackage[capitalize, nameinlink]{cleveref}
\crefformat{eq}{Equation~#2#1#3}
\crefformat{sec}{Section~#2#1#3}
\crefformat{def}{Definition~#2#1#3}
\crefformat{opt}{(#2#1#3)}
\crefformat{ineq}{Inequality~#2#1#3}
\crefformat{lem}{Lemma~#2#1#3}
\crefformat{cor}{Corollary~#2#1#3}
\crefformat{thm}{Theorem~#2#1#3}
\crefformat{alg}{Algorithm~#2#1#3}
\crefformat{claim}{Claim~#2#1#3}
\crefformat{prop}{Proposition~#2#1#3}

\usepackage{color}
\numberwithin{equation}{section}
\newcommand{\norm}[1] {\left \| #1 \right \|}
\newcommand{\dist}{\textrm{dist}}
\newcommand{\R}{\mathbb{R}}
\newcommand{\EE}{\mathbb{E}}
\newcommand{\conv}{\mathrm{conv}\,}
\global\long\def\defeq{\stackrel{\mathrm{def}}{=}}%

\newtheorem{thm}{Theorem}[section]
\newtheorem{lem}[thm]{Lemma}
\newtheorem{cor}[thm]{Corollary}

\newtheorem*{remark*}{Remark}
\newcommand{\declarecolor}[2]{\definecolor{#1}{RGB}{#2}\expandafter\newcommand\csname #1\endcsname[1]{\textcolor{#1}{##1}}}

\declarecolor{Black}{0, 0, 0}
\declarecolor{Blue}{91, 155, 213}
\declarecolor{LightBlue}{222, 235, 247}
\declarecolor{Green}{112, 173, 71}
\declarecolor{LightGreen}{226, 240, 217}
\declarecolor{Navy}{68, 114, 196}
\declarecolor{LightNavy}{218, 227, 243}

\hypersetup{
	colorlinks=true,
	pdfpagemode=UseNone,
	citecolor=Green,
	linkcolor=Navy,
	urlcolor=Black
}

\else 

\documentclass[11pt]{article} %for arxiv sub 
\title{
A gradient sampling method with complexity guarantees for  Lipschitz functions in high and low dimensions
} 
\date{}
\author{} 
\author{
			Damek Davis\thanks{School of ORIE, Cornell University, Ithaca, NY 14850, USA. \texttt{people.orie.cornell.edu/dsd95/}. Research of Davis supported by an Alfred P. Sloan research fellowship and NSF DMS award 2047637.}
			\and
			Dmitriy Drusvyatskiy\thanks{Department of Mathematics, U. Washington,
Seattle, WA 98195; \texttt{www.math.washington.edu/{\raise.17ex\hbox{$\scriptstyle\sim$}}ddrusv}. Research of Drusvyatskiy was supported by NSF DMS-1651851 and CCF-2023166 awards.}
			\and
			Yin Tat Lee\thanks{\texttt{yintat@uw.edu}. Paul G. Allen School of Computer Science and Engineering, U. Washington,
Seattle, WA 98195. Supported by NSF awards CCF-1749609, DMS-1839116, DMS-2023166, CCF-2105772, a Microsoft Research Faculty Fellowship, Sloan Research Fellowship, and Packard Fellowship.}
			\and
			Swati Padmanabhan\thanks{\texttt{pswati@uw.edu}. U. Washington, Seattle, WA 98195.}
			\and
		    Guanghao Ye\thanks{\texttt{ghye@mit.edu}. Department of Mathematics, Massachusetts Institute of Technology, Cambridge, MA 02139. Supported by an MIT Presidential Fellowship. Part of this work was done
			while the author was a student at University of Washington.} 
}
\makeatletter
\renewcommand{\@biblabel}[1]{[#1]\hfill}
\makeatother
\usepackage[margin=1in]{geometry}
\usepackage[font={small,bf}]{caption}

\usepackage[letterpaper, colorlinks,linkcolor=ForestGreen,citecolor=ForestGreen,
backref, bookmarks, bookmarksopen, bookmarksnumbered]{hyperref}

\usepackage{amsmath,amssymb, amsthm}    %
\usepackage{thmtools}

\usepackage[algo2e]{algorithm2e} %need the option when loading both algorithm and algorithm2e
\usepackage{algorithm,algpseudocode,algorithmicx}

\usepackage{color}
\definecolor{ForestGreen}{rgb}{0.1333,0.5451,0.1333}

\usepackage[numbers]{natbib}
\usepackage{url}
\numberwithin{equation}{section}

\usepackage[noabbrev,capitalize]{cleveref} 

\input{arxivmacros.tex}

\fi 

\begin{document}

\ifdefined\iscolt

\maketitle 
\begin{abstract}
\citet{pmlr-v119-zhang20p} introduced a novel modification of Goldstein's classical subgradient method, with  an efficiency guarantee of $O(\varepsilon^{-4})$ for minimizing Lipschitz functions. Their work, however, makes use of a nonstandard subgradient oracle model and requires the function to be directionally differentiable. In this paper, we show that both of these assumptions can be dropped by simply adding a small random perturbation in each step of their algorithm. The resulting method works on any Lipschitz  function whose value and gradient can be evaluated at points of differentiability. We additionally present a new cutting plane algorithm that achieves better efficiency in low dimensions: $O(d\varepsilon^{-3})$ for  Lipschitz functions and $O(d\varepsilon^{-2})$ for those that are weakly convex.
\end{abstract}

\else 

\begin{titlepage} 

\maketitle 
\begin{abstract}
Zhang et al.~\citep{pmlr-v119-zhang20p} introduced a novel modification of Goldstein's classical subgradient method, with  an efficiency guarantee of $O(\varepsilon^{-4})$ for minimizing Lipschitz functions. Their work, however, makes use of a nonstandard subgradient oracle model and requires the function to be directionally differentiable. In this paper, we show that both of these assumptions can be dropped by simply adding a small random perturbation in each step of their algorithm. The resulting method works on any Lipschitz  function whose value and gradient can be evaluated at points of differentiability. We additionally present a new cutting plane algorithm that achieves better efficiency in low dimensions: $O(d\varepsilon^{-3})$ for  Lipschitz functions and $O(d\varepsilon^{-2})$ for those that are weakly convex.
\end{abstract}

\thispagestyle{empty}
\end{titlepage}
\fi

\section{Introduction}
The subgradient method \citep{shor1985minimization} is a classical procedure for minimizing a nonsmooth Lipschitz function $f$ on $\RR^d$. Starting from an initial iterate $x_0$, the method computes
\begin{align}\label{eq:subgradientmethod}
x_{t+1} = x_t - \alpha_t v_t \,\text{ where }v_t \in \partial f(x_t).
\end{align}
Here, the positive sequence $\{\alpha_t\}_{t\geq 0}$ is user-specified, and the set $\partial f$ is the \emph{Clarke subdifferential},
$$\partial f(x)=\conv\left\{\lim_{i\to \infty} \nabla f(x_i): x_i\to x, ~x_i\in \dom(\nabla f)\right\}.$$
 In classical circumstances, the subdifferential reduces to familiar objects: for example, when $f$ is $C^1$-smooth at $x$, the subdifferential $\partial f(x)$ consists of only the gradient $\nabla f(x)$, while for convex functions, it reduces to the subdifferential in the sense of convex analysis. 
 
For general Lipschitz functions, the  process \eqref{eq:subgradientmethod} may fail to generate any meaningful limit points due to the existence of highly pathological examples \citep{daniilidis2020pathological}.
Nonetheless, for problems that are weakly convex or semialgebraic, the limit points $\bar x$ of the subgradient method are known to be first-order critical, meaning $0 \in \partial f(\bar x)$. Recall that a function $f$ is called $\rho$-weakly convex if the quadratically perturbed function $x\mapsto f(x)+\frac{\rho}{2}\|x\|^2$ is convex. In particular, convex and smooth functions are weakly convex.  Going beyond asymptotic guarantees, finite-time complexity estimates are known for smooth, convex, or weakly convex problems \citep{ghadimi2013stochastic, reddi2016stochastic, jin2017escape, allen2018make, carmon2018accelerated, davis2018subgradient, fang2018spider, zhou2018stochastic}. 

Modern machine learning, however, has witnessed the emergence of problems far beyond the weakly convex problem class. Indeed, tremendous empirical success has been recently powered by industry-backed solvers, such as Google's TensorFlow and Facebook's PyTorch, which routinely train nonsmooth nonconvex deep networks via (stochastic) subgradient methods. Despite a vast body of work on the asymptotic convergence of subgradient methods for nonsmooth nonconvex problems \citep{benaim2005stochastic, kiwiel2007convergence, majewski2018analysis, davis2020stochastic, bolte2021conservative}, no finite-time nonasymptotic convergence rates were known outside the weakly convex setting until recently, with 
\ifdefined\iscolt\citet{pmlr-v119-zhang20p}\else\citep{pmlr-v119-zhang20p} \fi  making a big leap forward towards this goal. 

In particular, restricting themselves to the class of Lipschitz and directionally differentiable functions, \ifdefined\iscolt\citet{pmlr-v119-zhang20p}\else\citep{pmlr-v119-zhang20p} \fi  developed an efficient algorithm motivated by Goldstein's conceptual subgradient method \ifdefined\iscolt\citet{goldstein1977optimization}\else\citep{goldstein1977optimization}\fi. 
Moreover, this was recently complemented by \citep{kornowski2021oracle} with lower bounds for finding \emph{near}-approximate-stationary points for nonconvex nonsmooth functions. 

One limitation of \citep{pmlr-v119-zhang20p} is that their complexity guarantees and algorithm use a nonstandard first-order oracle whose validity is unclear in examples. Our first contribution is to \emph{replace this assumption with a standard first-order oracle model}. We show (\cref{sec:INGD}) that a small modification of the algorithm of \ifdefined\iscolt\citet{pmlr-v119-zhang20p}\else\citep{pmlr-v119-zhang20p}\fi, wherein one simply adds a small random perturbation in each iteration, works for any Lipschitz function assuming only an oracle that can compute gradients and function values at almost every point of $\RR^d$ in the sense of Lebesgue measure. In particular, such oracles arise from automatic differentiation schemes routinely used in deep learning~\citep{bolte2020mathematical, bolte2021conservative}. Our end result is a randomized algorithm for minimizing any $L$-Lipschitz function that outputs a $(\delta,\epsilon)$-stationary point (Definition~\ref{def:GoldsteinDeltaEpsStationary}) after using at most $\widetilde{\mathcal O}\left(\frac{\Delta L^2}{\epsilon^3 \delta}\log(1/\gamma)\right)$ \footnote{Throughout the paper, we use $\widetilde{\mathcal{O}}({}\cdot{})$ to hide poly-logarithmic factors in $L,\delta,\Delta$, and $\epsilon$.} gradient and function evaluations. Here $\Delta$ is the initial function gap and $\gamma$ is the failure probability.

Having recovered the result of \ifdefined\iscolt\citet{pmlr-v119-zhang20p}\else\citep{pmlr-v119-zhang20p} \fi within the standard first-order oracle model, we then proceed to investigate the following question. 
\begin{quote}
\centering \emph{Can we improve the efficiency of the algorithm in {\bfseries low dimensions}?}
\end{quote}
In addition to being natural from the viewpoint of complexity theory, this question is well-grounded in applications. 
For instance, numerous problems in control theory involve minimization of highly irregular functions of a small number of variables. We refer the reader to the survey \citep[Section 6]{burke2020gradient} for an extensive list of examples, including
Chebyshev approximation by exponential sums, spectral and pseudospectral abscissa minimization, maximization of the “distance to instability”, and fixed-order controller design by static output feedback. We note that for many of these problems, the gradient sampling method of \cite{burke2020gradient} is often used. Despite its ubiquity in applications, the gradient sampling method does not have finite-time efficiency guarantees. The algorithms we present here offer an alternative approach with a complete complexity theory. 

The second contribution of our paper is \emph{an affirmative answer to the highlighted question}. We 
present a novel algorithm that uses $\widetilde{\mathcal O}\left(\frac{\Delta L d}{\epsilon^2 \delta}\log(1/\gamma)\right)$ calls to our (weaker) oracle. Thus we are able to trade off the factor $L\epsilon^{-1}$ with $d$. Further, if the function is $\rho$-weakly convex, the complexity improves to $\widetilde{\mathcal O}\left(\frac{\Delta d}{\epsilon \delta}\log(\rho)\right)$, which matches the complexity in $\delta=\epsilon$ of gradient descent for smooth minimization. Strikingly, the dependence on the weak convexity constant $\rho$ is only logarithmic.

The main idea underlying our improved dependence on $\epsilon$ in low dimensions is outlined next. The algorithm of \citep{pmlr-v119-zhang20p} comprises of an outer loop with ${\mathcal O}\left( \frac{\Delta}{\epsilon	\delta} \right)$ iterations, each performing either a decrease in the function value or an ingenious random sampling step to update the descent direction. Our observation, central to improving the $\varepsilon$ dependence, is that the violation of the descent condition can be transformed into a gradient oracle for the problem of finding a minimal norm element of the Goldstein subdifferential. This gradient oracle may then be used within a cutting plane method, which achieves better $\varepsilon$ dependence at the price of a dimension factor (\cref{sec:low-dimension}). 

\paragraph{Notation.}{Throughout, we let $\RR^d$ denote a $d$-dimensional Euclidean space equipped with a dot product $\langle \cdot,\cdot\rangle$ and the Euclidean norm $\|x\|_2=\sqrt{\langle x,x \rangle}$. The symbol $\mathbb{B}_r(x)$ denotes an open Euclidean ball of radius $r>0$ around a point $x$. Throughout, we fix a function $f\colon\R^d\to\R$ that is $L$-Lipschitz, and let $\dom(\nabla f)$ denote the set of points where $f$ is differentiable---a full Lebesgue measure set by Rademacher's theorem. The symbol $f'(x,u)\defeq\lim_{\tau\downarrow 0} \tau^{-1}(f(x+\tau u)-f(x))$ denotes the directional derivative of $f$ at $x$ in direction $u$, whenever the limit exists.}

\section{Interpolated Normalized Gradient Descent}\label[sec]{sec:INGD}
In this section, we describe the results in \citep{pmlr-v119-zhang20p} and our modified subgradient method that achieves finite-time guarantees in obtaining $(\delta, \epsilon)$-stationarity for an $L$-Lipschitz function $f:\R^d\rightarrow \R$. The main construction we use is the Goldstein subdifferential \citep{goldstein1977optimization}. 
\begin{definition}[Goldstein subdifferential]\label[def]{def:GoldsteinDeltaEpsStationary}
{\textrm Consider a locally Lipschitz function $f\colon\R^d\to\R$, a point $x\in\R^d$, and a parameter $\delta>0$. The {\em Goldstein subdifferential} of $f$ at $x$ is the set
$$\partial_{\delta} f(x)\defeq\conv \Big(\bigcup_{y\in \mathbb{B}_{\delta}(x)}\partial f(y)\Big).$$ A point $x$ is called $(\delta, \epsilon)$-stationary if $\dist(0,\partial_{\delta}f(x))\leq \epsilon$.
}
\end{definition}
Thus, the Goldstein subdifferential of $f$ at $x$ is the convex hull of all Clarke subgradients at points in a  $\delta$-ball around $x$. Famously, \ifdefined\iscolt\citet{goldstein1977optimization}\else\citep{pmlr-v119-zhang20p} \fi showed that one can significantly decrease the value of $f$ by taking a step in the direction of the minimal norm element of $\partial_{\delta} f(x)$. Throughout the rest of the section, we fix $\delta\in (0,1)$ and use the notation
$$\hat g\defeq g/\|g\|_2\, \textrm{ for any nonzero vector } g\in \R^d.$$ 

\begin{thm}[\ifdefined\iscolt\citet{goldstein1977optimization}\else\citep{goldstein1977optimization}\fi]\label[thm]{thm:gold_decr}
Fix a point $x$, and let $g$ be a minimal norm element of $\partial_{\delta} f(x)$. Then as long as $g\neq 0$, we have $f\left(x-\delta\hat g\right)\leq f(x)-\delta \|g\|_2$. 
\end{thm}

 \Cref{thm:gold_decr} immediately motivates the following conceptual descent algorithm:
\begin{equation}\label[eq]{eq:conceptualDescentAlg}
	x_{t+1}=x_t-\delta \hat g_t,\, \textrm{ where } g_t\in\argmin_{g\in \partial_{\delta} f(x)} \|g\|_2.
\end{equation}
In particular, \cref{thm:gold_decr} guarantees that, defining $\Delta\defeq f(x_0)-\min f$, the {\em approximate stationarity condition} 
 $$\min_{t=1,\ldots, T}~\|g_t\|_2\leq \epsilon\, \textrm{ holds after } T=\mathcal {O}\left(\frac{\Delta}{\delta\epsilon}\right)\, \textrm{ iterations of }~\eqref{eq:conceptualDescentAlg}.$$   
Evaluating the minimal norm element of $\partial_{\delta} f(x)$ is impossible in general, and therefore the descent method described in~\eqref{eq:conceptualDescentAlg} cannot be applied directly. Nonetheless it serves as a guiding principle for implementable algorithms. Notably, the gradient sampling algorithm \citep{burke2005robust} in each iteration forms polyhedral approximations $K_t$ of $\partial_{\delta} f(x_t)$ by sampling gradients in the ball $\mathbb{B}_{\delta}(x)$ and computes search directions $g_t\in \argmin_{g\in K_t} \|g\|_2$. These gradient sampling algorithms, however, have only asymptotic convergence guarantees \citep{burke2020gradient}. 

The recent paper \citep{pmlr-v119-zhang20p} remarkably shows that for any $x\in \R^d$ one \emph{can} find an \emph{approximate} minimal norm element of  $\partial_{\delta} f(x)$ using a number of subgradient computations that is independent of the dimension. The idea of their procedure is as follows. Suppose that we have a trial vector $g\in \partial_{\delta} f(x)$ (not necessarily a minimal norm element) satisfying
\begin{equation}\label[ineq]{eqn:approx_desce}
f\left(x-\delta\hat g\right)\geq f(x)-\frac{\delta}{2} \|g\|_2.
\end{equation}
That is, the decrease in function value is not as large as guaranteed by Theorem~\ref{thm:gold_decr} for the true minimal norm subgradient.
 One would like to now find a vector $u\in  \partial_{\delta} f(x)$ so that the norm of some convex combination $(1-\lambda)  g+\lambda u$ is smaller than that of $g$. A short computation shows that this is sure to be the case for all small $\lambda>0$ as long as $\langle u,g\rangle\leq \|g\|_2^2$. The task therefore  reduces to:
  $$\textrm{find some }u\in \partial_{\delta} f(x)\quad \textrm{satisfying}\quad \langle u,g\rangle\leq \|g\|_2^2.$$ The ingenious idea of \ifdefined\iscolt\citet{pmlr-v119-zhang20p}\else\citep{pmlr-v119-zhang20p} \fi is a randomized procedure for establishing exactly that in expectation. Namely, suppose for the moment that $f$ happens to be differentiable along the segment $[x,x-\delta \hat g]$; we will revisit this assumption shortly. Then the fundamental theorem of calculus, in conjunction with \eqref{eqn:approx_desce}, yields
 \begin{equation}\label[ineq]{eqn:fund_thm_smooth}
\frac{1}{2}\|g\|_2\geq \frac{f(x)-f\left(x-\delta\hat g\right)}{\delta}=\frac{1}{\delta}\int_0^\delta \langle \nabla f(x-\tau\hat g), \hat g\rangle ~d\tau.
\end{equation}
Consequently, a point $y$ chosen uniformly at random in the segment $[x,x-\delta \hat g]$ satisfies
\begin{equation}\label[ineq]{eqn:fund_thm_smooth2}
\EE \langle \nabla f(y),  g\rangle\leq \frac{1}{2}\|g\|_2^2.
\end{equation}
Therefore the vector $u=\nabla f(y)$ can act as the subgradient we seek. Indeed, the following lemma shows that, in expectation, the minimal norm  element of $[g,u]$ is significantly shorter than $g$. The proof is extracted from that of \cite[Theorem 8]{pmlr-v119-zhang20p}.

\begin{lem}[\ifdefined\iscolt\citet{pmlr-v119-zhang20p}\else\citep{pmlr-v119-zhang20p}\fi]\label[lem]{lem:dist_decr}	Fix a vector $g\in \R^d$, and let $u\in\R^d$ be a random vector satisfying $\EE\langle u,g\rangle< \frac{1}{2}\|g\|_2^2$. Suppose moreover that the inequality $\|g\|_2,\|u\|_2\leq L$ holds for some $L<\infty$. 
Then the minimal-norm vector $z$ in the segment $[g,u]$ satisfies:
$$\EE\|z\|_2^2\leq \|g\|_2^2-\frac{\|g\|_2^4}{16L^2}.$$
\end{lem}
\begin{proof}
Applying $\EE \langle u, g\rangle \leq \frac{1}{2}\|g\|^2_2$ and $\|g\|_2, \|u\|_2\leq L$, we have, for any $\lambda\in (0,1)$,
\begin{align*}
\EE\|z\|_2^2\leq\EE\|g+\lambda(u-g)\|_2^2&=\|g\|_2^2+2\lambda\EE\langle g,u-g\rangle+\lambda^2\EE\|u-g\|_2^2\\
%&\leq \|g\|_2^2-\lambda\|g\|_2^2+\lambda^2\EE\|u-g\|_2^2\\
&\leq\|g\|_2^2-\lambda\|g\|_2^2+4\lambda^2L^2.
\end{align*}
Plugging in the value $\lambda=\frac{\|g\|_2^2}{8L^2}\in (0,1)$ minimizes the right hand side and completes the proof.
\end{proof}

The last technical difficulty to overcome is the requirement that $f$ be differentiable along the line segment $[g,u]$. This assumption is crucially used to obtain \eqref{eqn:fund_thm_smooth} and \eqref{eqn:fund_thm_smooth2}. To cope with this problem, \ifdefined\iscolt\citet{pmlr-v119-zhang20p}\else\citep{pmlr-v119-zhang20p} \fi introduce extra assumptions on the function $f$ to be minimized and assume a nonstandard oracle access to subgradients.

We show, using Lemma~\ref{lem:generic_fubini}, that no extra assumptions are needed if one slightly perturbs $g$. 

\begin{lem}\label[lem]{lem:generic_fubini}
Let $f\colon\R^d\to\R$ be a Lipschitz function, and fix a point $x\in\R^d$. Then there exists a set $\mathcal{D}\subset \R^d$ of full Lebesgue measure such that for every $y\in \mathcal{D}$, the line spanned by $x$ and $y$ intersects  $\dom (\nabla f)$ in a  full Lebesgue measure set in $\R$. Then, for every $y\in \mathcal{D}$ and all $\tau\in\R$, we have
$$f(x+\tau(y-x))-f(x)=\int_{0}^{\tau}\langle \nabla f(x+s (y-x)),y-x \rangle~ds.$$
\end{lem}
\begin{proof}	
Without loss of generality, we may assume $x=0$ and $f(x)=0$. Rademacher's theorem guarantees that $\dom (\nabla f)$ has full Lebesgue measure in $\R^d$. Fubini's theorem then directly implies that there exists a set $\mathcal{Q}\subset \mathbb{S}^{d-1}$ of full Lebesgue measure  within the sphere $\mathbb{S}^{d-1}$ such that for every $y\in \mathcal{Q}$, the intersection $\R_+\{y\}\cap (\dom (\nabla f))^c$ is Lebesgue null in $\R$. It follows immediately that the set $\mathcal{D}=\{\tau y: \tau>0, y\in Q\}$ has full Lebesgue measure in $\R^d$.
Fix now a point $y\in \mathcal D$ and any $\tau\in\R_+$. Since $f$ is Lipschitz, it is absolutely continuous on any line segment and therefore
$$f(x+\tau(y-x))-f(x)=\int_{0}^{\tau} f'(x+s(y-x),y-x)~ds=\int_{0}^{\tau}\langle \nabla f(x+s (y-x)),y-x \rangle~ds.$$
The proof is complete.
\end{proof}

We now have all the ingredients to present a modification of the algorithm from \citep{pmlr-v119-zhang20p}, which, under a standard first-order oracle model, either significantly decreases the objective value or finds an approximate minimal norm element of $\partial_{\delta}f$.

\begin{algorithm}[!ht]
	\caption{$\mathtt{MinNorm}(x)$}\label[alg]{alg:approx_desc}
	\KwIn{$x$, $\delta >0$, and $\epsilon>0$.}
	
	 Let $k=0$, $g_0 = \nabla f(\zeta_0)$ where $\zeta_0 \sim \mathbb{B}_{\delta}(x)$.\; 
	 
	\While {$\|g_k\|_2>\epsilon$ and $\frac{\delta}{4} \|g_k\|_2\geq f(x)-f\left(x-\delta\hat g_k\right)$}{
	
		Choose any $r$ satisfying $ 0<r<\|g_k\|_2\cdot\sqrt{1-(1-\tfrac{\|g_k\|_2^2}{128L^2})^2}$. \;
		 
		Sample $\zeta_k ~\textrm{uniformly~from}~\mathbb{B}_{r}(g_k)$.\;
		
		Choose $y_k$ uniformly at random from the segment $ [x, x-\delta \widehat{\zeta}_k]$.\; 
		
		$g_{k+1}=\argmin_{z\in [g_k,\nabla f(y_k)]} \|z\|_2$.\;
		
		$k=k+1$.\; 
		
	}
	Return $g_k$.\; 
	
\end{algorithm}

The following theorem establishes the efficiency of Algorithm~\ref{alg:approx_desc}, and its proof is a small modification of that of \cite[Lemma 13]{pmlr-v119-zhang20p}.

\begin{thm}\label[thm]{thm:minnorm}
Let $\{g_k\}$ be generated by $\mathtt{MinNorm}(x)$. Fix an index $k\geq 0$, and define the stopping time $\tau \defeq \inf\left\{ k \colon f(x - \delta \hat g_k) < f(x) - \delta \| g_k\|_2/4\, \text{ or  } \|g_k\|_2 \leq \epsilon\right\}$. Then,  we have
$$\EE\left[\|g_{k}\|_2^2  1_{\tau > k}\right]\leq \frac{16L^2}{16+k}.$$
\end{thm}
\begin{proof}
	Fix an index $k$, and let $\mathbb{E}_k[\cdot]$ denote the conditional expectation on $g_k$.
Suppose we are in the event $\{ \tau >k\}$. 
Taking into account the Lipschitz continuity of $f$ and Lemma~\ref{lem:generic_fubini}, we deduce that almost surely, conditioned on $g_k$, the following estimate holds: 
	\begin{align*}
\frac{1}{4}\|g_k\|_2\geq \frac{f(x)-f\left(x-\delta\hat g_k\right)}{\delta}&\geq \frac{f(x)-f(x-\delta\cdot \hat \zeta_k)}{\delta}-L\|\hat g_k-\hat \zeta_k\|_2\\
&=\frac{1}{\delta}\int_0^\delta \langle \nabla f(x-s\hat \zeta_k), \hat \zeta_k\rangle ~ds-L\|\hat g_k-\hat\zeta_k\|_2\\
&\geq\frac{1}{\delta} \int_0^\delta \langle \nabla f(x-s\hat \zeta_k), \hat g_k\rangle ~ds-2L\|\hat g_k-\hat\zeta_k\|_2\\
&= \EE_k\langle \nabla f(y_k),\hat g_k\rangle- 2L\|\hat g_k-\hat\zeta_k\|_2.
\end{align*}
Rearranging yields 
$\EE_k\langle \nabla f(y_k),\hat g_k\rangle\leq \frac{1}{4}\|g_k\|_2+2L\|\hat g_k-\hat \zeta_k\|$. Simple algebra shows $\|\hat g_k-\hat \zeta_k\|_2^2\leq 2(1-\sqrt{1-r^2/\|g_k\|_2^2})\leq \frac{\|g_k\|_2^2}{64L^2}$. Therefore, we infer that $\EE_k\langle \nabla f(y_k),\hat g_k\rangle< \frac{1}{2}\|g_k\|_2$. 
Lemma~\ref{lem:dist_decr} then guarantees that
$$\EE_k[\|g_{k+1}\|_2^2 1_{\tau > k}]\leq \left(\|g_k\|_2^2-\frac{\|g_k\|_2^4}{16L^2}\right)1_{\tau > k}.$$
Define $b_k := \|g_{k}\|_2^21_{\tau > k}$ for all $k \geq 0$. Then the tower rule for expectations yields
$$\EE b_{k+1}\leq \EE[\|g_{k+1}\|_2^2 1_{\tau > k}] \leq \EE\left[\left(1-\frac{b_k}{16L^2}\right)b_k\right]\leq \left(1-\frac{\EE b_k}{16L^2}\right)\EE b_k,$$
by Jensen's inequality applied to the concave function $t\mapsto (1-t/16L^2)t$. Setting $a_k=\EE b_k/L^2$, this inequality becomes $a_{k+1}\leq a_k-a_k^2/16$, which, upon rearranging, yields $\frac{1}{a_{k+1}}\geq \frac{1}{a_k(1-a_k/16)}\geq \frac{1}{a_k}+ \frac{1}{16}$. Iterating the recursion and taking into account $a_0\leq 1$ completes the proof.
\end{proof}

An immediate consequence of \cref{thm:minnorm} is that $\mathtt{MinNorm}(x)$ terminates with high-probability.
\begin{cor}\label[cor]{cor:MinNorm} $\mathtt{MinNorm}(x)$  terminates in at most $\left\lceil \frac{64L^2}{\epsilon^2}\right\rceil \cdot \left\lceil 2\log(1/\gamma)\right\rceil$ iterations with probability at least $1 - \gamma$. 
\end{cor}
\begin{proof} 
Notice that when $k \geq \frac{64L^2}{\varepsilon^2}$, we have, by \cref{thm:minnorm}, that 
$$
\textrm{Pr}(\tau > k) \leq \textrm{Pr}(\|g_k\|_21_{\tau > k} \geq \epsilon) \leq \frac{16L^2}{(16 + k)\varepsilon^2} \leq \frac{1}{4}.
$$
Similarly, for all $i \in \NN$, we have $\textrm{Pr}(\tau > ik \mid \tau > (i-1)k) \leq 1/4$. Therefore, 
\begin{align*}
\textrm{Pr}(\tau > ik) &= \textrm{Pr}(\tau > ik \mid \tau > (i-1)k)\textrm{Pr}(\tau > (i-1)k) \leq \frac{1}{4}\textrm{Pr}(\tau > (i-1)k) \leq \frac{1}{4^i}.
\end{align*}
Consequently, we have $\textrm{Pr}(\tau > ik) \leq \frac{1}{4^i} \leq \gamma$ 
whenever $i \geq \log(1/\gamma)/\log(4)$, as desired.
\end{proof}

Combining \cref{alg:approx_desc} with \eqref{eq:conceptualDescentAlg} yields Algorithm~\ref{alg:INGD}, with convergence guarantees summarized in \cref{thm:finalZhangSimplified}, whose proof is identical to that of  \cite[Theorem 8]{pmlr-v119-zhang20p}.

\begin{algorithm}[!ht]
\DontPrintSemicolon
\caption{Interpolated Normalized Gradient Descent ($\mathtt{INGD}(x_0,T))$}\label[alg]{alg:INGD}
\KwIn{Initial $x_0$, counter $T$}

\For{$t=0,\ldots, T-1$}{

$\qquad\;  g=\mathtt{MinNorm}(x_t)$ \tcp*{Computational complexity $\widetilde{\mathcal{O}}(L^2/\epsilon^2)$} \; 

$\qquad$ Set $x_{t+1}=x_t-\delta\hat g$\; 

}

\textbf{Return} $x_{T}$\; 

\end{algorithm}

\begin{thm}\label[thm]{thm:finalZhangSimplified}
Fix an initial point $x_0\in \R^d$, and define $\Delta=f(x_0)-\inf_x f(x)$.
Set the number of iterations $T=\frac{4\Delta}{\delta\epsilon}$. Then, with probability $1-\gamma$, the point $x_T=\mathtt{INGD}(x_0,T)$ satisfies $\dist(0,\partial_{\delta}f(x_T))\leq \epsilon$ in a total of at most  
$$\left\lceil\frac{4\Delta}{\delta\epsilon} \right\rceil \cdot \left\lceil \frac{64L^2 }{\epsilon^2}\right\rceil \cdot \left\lceil 2\log\left(\frac{4\Delta}{\gamma\delta\epsilon}\right)\right\rceil\qquad \text{ function-value and gradient evaluations}.$$
\end{thm}

In summary, the complexity of finding a point $x$ satisfying $\dist(0,\partial_{\delta}f(x))\leq \epsilon$ is at most $\mathcal{O}\left(\frac{\Delta L^2 }{\delta\epsilon^3}\log\left(\frac{4\Delta}{\gamma\delta\epsilon}\right)\right)$ with probability $1-\gamma$. Using the identity $\partial f(x) = \limsup_{\delta \rightarrow 0} \partial_{\delta} f(x)$, this result also provides a strategy for finding a Clarke stationary point, albeit with no complexity guarantee. 
It is thus natural to ask whether one may efficiently find some point $x$ for which there exists $y\in \mathbb{B}_{\delta}(x)$ satisfying $\dist(0,\partial f(y))\leq \epsilon$. This is exactly the guarantee of subgradient methods on weakly convex functions in~\citep{davis2019stochastic}. \ifdefined\iscolt\citet{shamir2020can}\else\citep{shamir2020can} \fi  shows that for general Lipschitz functions, the number of subgradient computations required to achieve this goal by any algorithm scales with the dimension of the ambient space.
Finally, we mention that the perturbation technique similarly applies to the stochastic algorithm of~\cite[Algorithm 2]{pmlr-v119-zhang20p}, yielding a method that matches their complexity estimate.

\section{Faster INGD in Low Dimensions}\label[sec]{sec:low-dimension}
\global\long\def\y{\mathbf{y}}%
\global\long\def\x{\mathbf{x}}%
\global\long\def\d{\mathbf{d}}%
\global\long\def\g{\mathbf{g}}%
\global\long\def\u{\mathbf{u}}%
\global\long\def\z{\mathbf{z}}%
\global\long\def\mt{\mathbf{m}_{t}}%
\global\long\def\m{\mathbf{m}}%

\global\long\def\Rn{\mathbb{R}^{n}}%
\global\long\def\R{\mathbb{R}}%
\global\long\def\inprod#1#2{\langle#1,#2\rangle}%
\global\long\def\E{\mathbb{E}}%
\global\long\def\norm#1{\|#1\|}%

\global\long\def\xtk{\mathbf{x}_{t,k}}%
\global\long\def\xt{\mathbf{x}_{t}}%
\global\long\def\mtk{\mathbf{m}_{t,k}}%
\global\long\def\xtp{\x_{t+1}}%
\global\long\def\xtpl{\x_{t}^{+}}%
\global\long\def\yt{\y_{t}}%
\global\long\def\mtkp{\m_{t,k+1}}%

\global\long\def\ball{\mathcal{B}}%
\global\long\def\halfspace{\mathcal{H}}%
\global\long\def\A{\mathcal{A}}%
\global\long\def\dist{\text{dist}}%
\global\long\def\tr{\mathrm{Tr}}%
\global\long\def\fdir#1#2{f^{\circ}(#1;#2)}%
\global\long\def\ft{f_{t}}%

\global\long\def\gs{\g^{\star}}%
\global\long\def\T{\mathcal{T}}%
\global\long\def\ut{\mathbf{u}_{t}}%
\global\long\def\v{\mathbf{v}}%
\global\long\def\w{\mathbf{u}_{t+1}}%
\global\long\def\hg{\widehat{\g}}%
\global\long\def\vs{\v^{\star}}%
\global\long\def\O{\mathcal{O}}%

\global\long\def\fs{\bar{f}}%
\global\long\def\ts{t^{\star}}%
\global\long\def\gph{G}%
 
\global\long\def\dd{\mathrm{d}}%
 
\global\long\def\dist{\operatorname{dist}}%
\global\long\def\eps{\epsilon}%
\global\long\def\fse{\bar{f}_{\epsilon}}%
\global\long\def\fbp{\overline{f}^{\prime}}%
\global\long\def\conv{\operatorname{conv}}%
\global\long\def\oracle{\mathscr{O}}%
\global\long\def\vol{\mathrm{vol}}%
In this section, we describe our modification of \cref{alg:approx_desc} for obtaining improved runtimes in the low-dimensional setting. Our modified algorithm hinges on computations similar to \eqref{eqn:approx_desce}, \eqref{eqn:fund_thm_smooth}, and \eqref{eqn:fund_thm_smooth2} except for the constants involved, and hence we explicitly state this setup. Given a vector $g\in \partial_{\delta} f(x)$, we say
it satisfies the \emph{descent condition }at $x$ if 
\begin{equation}
f(x-\delta\hat{g})\leq f(x)-\frac{\delta\epsilon}{3}.\label[eq]{eq:descent-condition}
\end{equation}
Recall that \cref{lem:generic_fubini} shows that for almost all $g$, we have 
\begin{align*}
f(x)-f(x-\delta\hat{g}) & =\int_{0}^{1}\langle\nabla f(x-t\delta\hat{g}),\hat{g})\ dt =\delta\cdot\E_{z\sim\text{Unif}[x-\delta\hat{g},x]}\langle\nabla f(z),\hat{g}\rangle.
\end{align*}Hence, when $g$ does \emph{not} satisfy the descent condition~\eqref{eq:descent-condition}, we can output a random vector $u\in\partial_{\delta}f(x)$ such that
\begin{equation}
\mathbb{E}\langle u,g\rangle\leq\frac{\epsilon}{3}\|g\|_2.\label[ineq]{eq:expect_ug}
\end{equation}
Then, an arbitrary vector $g$ either satisfies
 \eqref{eq:descent-condition} or can be used to output a random vector $u$ satisfying
\eqref{eq:expect_ug}. As described in \cref{cor:MinNorm},  \cref{alg:approx_desc} achieves this goal in $\widetilde{\mathcal{O}}(L^{2}/\epsilon^{2})$ iterations. 

In this section, we improve upon this oracle complexity by applying cutting plane methods to design~\cref{alg:approx-min-norm}, which finds a better descent direction in $\widetilde{\mathcal{O}}(Ld/\epsilon)$ oracle calls for $L$-Lipschitz functions and $\mathcal{O}(d \log(L/\epsilon)\log (\delta \rho/\epsilon))$ oracle calls for $\rho$-weakly convex  functions. In \cref{sec:oracle}, we demonstrate
how to remove the expectation in \eqref{eq:expect_ug} and turn the
inequality into a high probability statement. For now, we
assume the existence of an oracle $\mathscr{O}$ as in \cref{def-innerprod-oracle}. 
\begin{definition}[Inner Product Oracle]\label[def]{def-innerprod-oracle}
Given a vector $g\in \partial_{\delta}f(x)$ that does not satisfy the descent condition
\eqref{eq:descent-condition}, the inner product oracle $\oracle(g)$ outputs
a vector $u\in\partial_{\delta}f(x)$ such that 
\[
\inprod ug \leq\frac{\epsilon}{2}\|g\|_2.
\]
\end{definition}
We defer the proof of the lemma below to \cref{sec:oracle}.
\begin{lem}
\label[lem]{lem:inner-product-oracle}
Fix  $x\in \R^d$ and a unit vector $\hat g\in \RR^d$ such that $f$ is differentiable almost everywhere on the line segment $[x,y]$, where 
$y\defeq x-\delta\hat g$. Suppose that $z\in \R^d$ sampled uniformly from $[x,y]$ satisfies $\mathbb{E}_{z}\langle\nabla f(z),\hat{g}\rangle\leq\frac{\epsilon}{3}$.
Then we can find
$\bar{z}\in \R^d$ using at most $O(\frac{L}{\epsilon}\log(1/\gamma))$ gradient evaluations of $f$, such that with probability at least $1-\gamma$ the estimate $\langle\nabla f(\bar{z}),\hat{g}\rangle\leq\frac{\epsilon}{2}$
holds. Moreover, if $f$ is $\rho$-weakly convex, we can find $\bar{z}\in \R^d$
such that $\langle\nabla f(\bar{z}),\hat{g}\rangle\leq\frac{\epsilon}{2}$
using only $O(\log(\delta \rho/\epsilon))$ function evaluations of $f$.
\end{lem}

Our key insight is that this oracle is almost identical to the
gradient oracle of the minimal norm element problem
\[
\min_{g\in\partial_{\delta}f(x)}\|g\|_2.
\]
Therefore, we can use it in the cutting plane method to find an approximate
minimal norm element of $\partial_{\delta}f$. When there is no element
of $\partial_{\delta}f$ with norm less than $\epsilon$, our algorithm
will instead find a vector that satisfies the descent condition. The main result of this section is the following theorem. 
\begin{thm}\label[thm]{thm:main-theorem}Let $f:\R^{d}\to\R$ be an $L$-Lipschitz
   function. Fix an initial point $x_{0}\in\R^{d}$,
   and let $\Delta\defeq f(x_{0})-\inf_{x}f(x)$. Then, there exists
   an algorithm that outputs a point $x\in\R^{d}$ satisfying $\dist(0,\partial_{\delta}f(x))\leq\epsilon$ and,
    with probability at least $1-\gamma$,  uses at most
   $$\mathcal O\left(\frac{\Delta Ld}{\delta\eps^{2}}\cdot\log(L/\epsilon)\cdot\log(1/\gamma)\right)\qquad \text{ function value/gradient evaluations.}$$  If $f$ is
   $\rho$-weakly convex, the analogous statement holds with probability one and with the improved efficiency estimate $\mathcal{O}\text{\ensuremath{\left(\frac{\Delta d}{\delta\epsilon}\log(L/\epsilon)\cdot\log(\delta \rho/\epsilon)\right)}}$ of
    function value/gradient evaluations.
   \end{thm}

  % Maybe add a remark saying that this algorithm can also add the noise?

\subsection{Finding a Minimal Norm Element}\label[sec]{sec:MinNorm}

In this section, we show, via \cref{alg:approx-min-norm}, how to find an approximate minimal norm element of $\partial_{\delta} f(x)$. Instead
of directly working with the minimal norm problem, we note that, by
Cauchy-Schwarz inequality and the Minimax Theorem, for any closed convex
set $Q$, we have
\begin{equation}
\min_{g\in Q}\|g\|_{2}=\min_{g\in Q}\left[\max_{\|v\|_2\le1}\inprod gv\right]=\max_{\|v\|_2\le1}\left[\min_{g\in Q}\inprod gv\right]=\max_{\|v\|_2\leq1}\phi_{Q}(v),\label[eq]{minimaxMagic}
\end{equation}
where $\phi_{Q}(v)\defeq\min_{g\in Q}\inprod gv$, and \cref{lem:vstar} formally connects the problem of finding the minimal norm element with that
of maximizing $\phi_Q$. %For $Q=\partial_{\delta}f(x)$,
%and for some given $v$, the gradient of $\phi_Q$ is exactly some gradient $g$ of $f$ that
%minimizes $\inprod gv$. Hence, one can think any gradient $g$ with
%small $\inprod gv$ as an ``approximate gradient''. 
The key observation  in this section
(Lemma~\ref{lem:separationOracle}) is that the inner
product oracle $\oracle$ is a separation oracle for the (dual) problem $\max_{\|v\|_2\leq1}\phi_{Q}(v)$ with $Q=\partial_{\delta}f(x)$ and hence can be used in cutting
plane methods.
\begin{lem}
\label[lem]{lem:vstar} Let $Q\subset\RR^d$ be a closed convex set that does not contain the origin. Let $g_{Q}^{*}$ be a minimizer of $\min_{g\in Q}\|g\|_{2}$.
Then, the vector $v_{Q}^{*}=g_{Q}^{*}/\|g_{Q}^{*}\|_{2}$ satisfies 
\[
\langle v_{Q}^{*}, g\rangle\geq\|g_{Q}^{*}\|_{2}\qquad\text{ for all }g\in Q.
\]
and $v_{Q}^{*}=\arg\max_{\|v\|_{2}\leq1}\phi_{Q}(v).$
\end{lem}
\begin{proof}
We omit the subscript $Q$ to simplify notation. Since, by definition, $g^*$ minimizes $\|g\|_2$ over all $g\in Q$, we have 
\[
\langle g^{*}, g \rangle \geq\|g^{*}\|_{2}^{2}\text{ for all }g\in Q,
\]
and the inequality is tight for $g=g^{*}.$ Using this fact and $\phi(v^{*})=\min_{g\in Q} \langle g, \frac{g^{*}}{\|g^{*}\|_2} \rangle$ gives
\[
\phi(v^*) =\|g^{*}\|_{2}=\min_{g\in Q}\|g\|_{2}=\min_{g\in Q}\max_{v:\|v\|_{2}\leq1} \inprod gv =\max_{\|v\|_{2}\leq1}\min_{g\in Q}\inprod gv=\max_{v:\|v\|_{2}\leq1}\phi(v),
\]
where we used Sion's minimax theorem in the second to last step. This completes
the proof. 
\end{proof}
Using this lemma, we can show that $\oracle$ is a separation
oracle.
\begin{lem}
\label[lem]{lem:separationOracle} Consider a vector $g\in \partial f_{\delta}(x)$ that does not satisfy the descent condition \eqref{eq:descent-condition}, and let the output of querying the oracle at $g$ be 
$u\in \oracle( g)$. Suppose that $\dist(0,\partial_{\delta} f(x))\geq\frac{\epsilon}{2}$.
Let $g^*$ be the minimal-norm element of $\partial_{\delta} f(x)$. Then the normalized vector $v^*\defeq g^*/\|g^*\|_2$ satisfies the inclusion:
\[
v^{*}\in\left\{w\in\R^{d}:\langle u,\hat g-w\rangle\leq0\right\}.
\]
\end{lem}
\begin{proof}
Set $Q=\partial_{\delta}f(x)$. By using $\langle u,\hat g\rangle\leq\frac{\epsilon}{2}$ (the guarantee of $\oracle$ per \cref{def-innerprod-oracle}) and $\langle u,v^* \rangle \geq \|g^*\|_2$ (from \cref{lem:vstar}), we have $\langle u, \hat g-v^* \rangle=\langle u,\hat g\rangle-\langle u,v^* \rangle\leq\frac{\epsilon}{2}-\|g^{*}\|_{2}\leq0$.  
\end{proof}

Thus Lemma~\ref{lem:separationOracle}
states that if $x$ is not a $(\delta,\frac{\epsilon}{2})$-stationary point of $f$, then 
the oracle $\oracle$ produces a halfspace $\mathcal{H}_{v}$ that
separates $\hat g$ from $v^{*}$.
Since $\oracle$ is a separation oracle, we can combine
it with any cutting plane method to find $v^{*}$. For concreteness,
we use the center of gravity method and display our algorithm in \cref{alg:approx-min-norm}. 
%\begin{thm}[\citet{grunbaum1960partitions}]\label[thm]{lem:Grunbaum}Given any convex set $K$ and any halfspace $H$
%containing the center of gravity of $K$, we have
%\[
%\vol(K\cap H)\leq(1-1/e)\vol(K).
%\]
%\end{thm}
%
%
Note that in our algorithm, we use a point $\zeta_k$ close to the true center of gravity of $\Omega_k$, and therefore, we invoke a result about the \emph{perturbed} center of gravity method.
\begin{thm}[Theorem 3 of \cite{DBLP:journals/jacm/BertsimasV04}; see also \citep{grunbaum1960partitions}]\label[thm]{thm:Grunbaum-perturbation}
	Let $K$ be a convex set with center of gravity $\mu$ and covariance matrix $A$. For any halfspace $H$ that contains some point $x$ with $\|x-\mu\|_{A^{-1}}\leq t$, we have
	\[
		\vol(K\cap H)\leq(1-1/e+t)\vol(K).
	\]
\end{thm}

\begin{thm}[Theorem 4.1 of \cite{10.1007/BF02574061}]\label[thm]{thm:circumscribed-ellipsoid}
	Let $K$ be a convex set in $\R^d$ with center of gravity $\mu$ and covariance matrix $A$. Then,
	\[
		K\subset \left\{x:\|x-\mu\|_{A^{-1}}\leq \sqrt{d(d+2)}\right\}.
	\]
\end{thm}

We now have all the tools to show correctness and iteration complexity of \cref{alg:approx-min-norm}. 

\begin{algorithm}
	\caption{$\mathtt{MinNormCG}(x)$}\label[alg]{alg:approx-min-norm}
		\KwIn{center point $x$.}
		
		Set $k=0$, the search region $\Omega_0=\mathbb{B}_2(0)$, the set of gradients $Q_0=\{\nabla f(x)\}$, and $r$ satisfying $0<r< \epsilon / (32 d L)$\;
		
		\While {$\min_{g\in Q_k}\|g\|_2>\epsilon$}{
		Let $v_k$ be the center of gravity of $\Omega_k$.\;
	
		\If { $v_k$ satisfies the descent condition \eqref{eq:descent-condition} at $x$}{
			Return $v_k$\;
    	}
    			
		Sample $\zeta_k ~\textrm{uniformly~from}~\mathbb{B}_{r}(v_k)$\;

		$u_k\leftarrow\oracle(\zeta_k)$\; 

		$\Omega_{k+1}=\Omega_{k}\cap\left\{w:\langle u_k,\zeta_k-w\rangle\leq0\right\}$.\;

		$Q_{k+1}=\conv(Q_k\cup\{u_k\})$\;

		$k=k+1$\; 
		}
		
		Return $\arg\min_{g\in Q_k}\|g\|_2$.\;
\end{algorithm}
\begin{thm}
\label[thm]{thm:minimal-norm-CPM} Let $f:\R^{d}\to\R$ be an $L$-Lipschitz
   function. Then \cref{alg:approx-min-norm} returns a vector $v\in\partial_{\delta}f(x)$ that either satisfies the descent
condition~\eqref{eq:descent-condition} at $x$ or satisfies $\|v\|_2\leq\epsilon$ in $$\lceil 8d\log(8L/\eps)) \rceil  \text{ calls to $\oracle$}.$$
\end{thm}
\begin{proof}
By the description
of \cref{alg:approx-min-norm}, either it returns a vector $v$ satisfying the descent condition
or returns $g\in\partial_{\delta}f(x)$ with $\|g\|_2\leq\epsilon$. We now obtain the algorithm's claimed iteration complexity.

Consider an iteration $k$ such that $\Omega_k$ \emph{does} contain a ball of radius $\frac{\epsilon}{4L}$. Let $A_k$ be the covariance matrix of convex set $\Omega_k$. By \cref{thm:circumscribed-ellipsoid}, we have \[A_k \succeq \left(\frac{\epsilon}{8d L}\right)^2 I.\] Applying this result to the observation that in \cref{alg:approx-min-norm} $\zeta_k$ is sampled uniformly from $\mathbb{B}_r(v_k)$ gives 
\begin{equation}\label{eq:proximityCondition}
\|v_k - \zeta_k\|_{A_k^{-1}} \leq r \cdot \frac{8dL}{\epsilon}\leq \frac{1}{4}.
\end{equation} Recall from \cref{alg:approx-min-norm} and the preceding notation that $\Omega_k$ has center of gravity $v_k$ and covariance matrix $A_k$. Further, the halfspace $\left\{w: \langle u_k, \zeta_k - w\rangle \leq 0\right\}$ in \cref{alg:approx-min-norm} contains the point $\zeta_k$ satisfying \eqref{eq:proximityCondition}. Given these statements, since \cref{alg:approx-min-norm} sets $\Omega_{k+1} = \Omega_k \cap \left\{w: \langle u_k, \zeta_k - w \rangle \right\}$, we may invoke
%By \cref{alg:approx-min-norm}, the set $\Omega$ is replaced by $\Omega \cap H$ for some halfspace $H$ that contains a point $x$ that is close to the center of gravity $\mu$. More specifically, we have $\|x-\mu\|_{A^{-1}}\leq r\cdot \frac{8dL}{\epsilon} \leq \frac{1}{4}$. 
\cref{thm:Grunbaum-perturbation} to obtain
\begin{align}
	\vol(\Omega_{k})&\leq(1-1/e + 1/4)^{k}\vol(\mathbb{B}_2(0))\leq (1-1/10)^k \vol(\mathbb{B}_2(0)).  \label{eq:volOmegaKBall}
\end{align} We claim that \cref{alg:approx-min-norm} takes at most $T+1$
steps where $T=d\log_{(1-\frac{1}{10})}(\epsilon/(8L))$. For the sake of contradiction, suppose that 
this statement is false. Then, applying \eqref{eq:volOmegaKBall} with $k = T+1$  gives
\begin{equation}\label{eq:OmegaSmall}
\vol(\Omega_{T+1})\leq \left(\frac{\epsilon}{4L}\right)^{d}\vol(\mathbb{B}_1(0)).
\end{equation} On the other hand, \cref{alg:approx-min-norm} generates points $u_i = \oracle(\zeta_{i})$ in the $i$-th call to $\oracle$ and the set $Q_{i}=\conv\left\{u_{1},u_{2},\cdots,u_{i}\right\}$.
Since we assume that the algorithm takes more than $T+1$ steps, we have $\min_{g\in Q_{T+1}}\|g\|_2\geq\epsilon$.
Using this and $u_{i}\in Q_{T+1}$, \cref{lem:separationOracle} lets us conclude
that $v_{Q_{T+1}}^{*}\in\left\{w\in\R^{d}:\langle u_{i},\zeta_{i}-w\rangle\leq0\right\}\text{ for all }i\in[T+1].$ Since $\Omega_{T+1}$ is the intersection of the unit ball and these
halfspaces, we have \[v_{Q_{T+1}}^{*}\in\Omega_{T+1}. \] Per \eqref{eq:OmegaSmall}, $\Omega_{T+1}$ does not contain a ball of radius $\frac{\epsilon}{4L}$, and therefore
we may conclude that \[ \text{ there exists a point } \widetilde{v}\in \mathbb{B}_{\frac{\epsilon}{2L}}({v_{Q_{T+1}}^{*}})\, \text{ such
that }\widetilde{v}\notin\Omega_{T+1}.\] Since $\widetilde{v}\in\mathbb{B}_2(0)$,
the fact $\widetilde{v}\notin\Omega_{T+1}$ must be true due to one of the
halfspaces generated in \cref{alg:approx-min-norm}. In other words, there must exist some $i\in[T+1]$ with
\[
\langle u_{i},\zeta_{i}-\widetilde{v}\rangle>0.
\]
By the guarantee of $\oracle$, we have $\langle u_{i},\zeta_{i}\rangle\leq\frac{\epsilon}{2}$,
and hence
\begin{equation}\label{eq:final1}
\langle u_{i},\widetilde{v}\rangle=\langle u_{i},v_{i}\rangle-\langle u,v_{i}-\widetilde{v}\rangle<\frac{\epsilon}{2}.
\end{equation}
By applying $\widetilde{v}\in \mathbb{B}_{\frac{\epsilon}{2L}}(v_{Q_{T+1}}^{*})$, $u_{i}\in\partial_{\delta}f(x)$, $L$-Lipschitzness of $f$, and \cref{lem:vstar}, we have 
\begin{equation}\label{eq:final2}
\langle u_{i}, \widetilde{v}\rangle  \geq \langle u_{i}, v_{Q_{T+1}}^{*} \rangle -\frac{\epsilon}{2L}\|u_{i}\|_{2} \geq \langle u_{i}, v_{Q_{T+1}}^{*}\rangle -\frac{\epsilon}{2} \geq\|g_{Q_{T+1}}^{*}\|_{2}-\frac{\epsilon}{2}.
\end{equation}
Combining \eqref{eq:final1} and \eqref{eq:final2} yields
that $\min_{g\in Q_{T+1}}\|g\|_2=\|g_{Q_{T+1}}^{*}\|_{2}<\epsilon$. This contradicts the assumption that the algorithm takes
more than $T+1$ steps and concludes the proof. 
\end{proof}

\iffalse % Not needed now because we can simply use the INGD algorithm before
\begin{algorithm}
\caption{Interpolated Normalized Gradient Descent }
\label{alg:INGD}

\begin{algorithmic}[1]

\algnewcommand{\LeftComment}[1]{\State \(\triangleright\) #1} 

\Procedure{$\textsc{INGD}$}{$x_{0}$}

\State $t=0$

\Repeat

	\State Let $v=\textsc{MinNorm}(x_{t})$

	\State $x_{t+1}\leftarrow x_{t}+\delta\frac{v}{\|v\|_2}.$

	\State $t\leftarrow t+1$

\Until{ $\|v\|_2<\eps$}

\State \Return $x_{t-1}.$

\EndProcedure

\end{algorithmic}
\end{algorithm}
\fi

Now, we are ready to prove the main theorem.
\begin{proof}[Proof of~\cref{thm:main-theorem}]
 We note that the outer loop in~\cref{alg:INGD} runs at most $\mathcal{O}(\frac{\Delta}{\delta\epsilon})$
times because we decrease the objective by $\Omega(\delta\epsilon)$
every step. Combining this with \cref{thm:minimal-norm-CPM,lem:inner-product-oracle},
we have that with probability $1 - \gamma$, the oracle complexity for $L$-Lipschitz function is $$\left\lceil\frac{4\Delta}{\delta\epsilon} \right\rceil \cdot 
	\left\lceil 8d\log(8L/\eps)) \right \rceil \cdot
	\left\lceil \frac{36L}{\epsilon}\right\rceil \cdot
	 \left\lceil 2\log\left(\frac{4\Delta}{\gamma\delta\epsilon}\right)\right\rceil=\mathcal O\left(\frac{\Delta Ld}{\delta\eps^{2}}\cdot\log(L/\epsilon)\cdot\log(1/\gamma)\right)$$ and for $L$-Lipschitz and $\rho$-weakly convex function is $\mathcal O\text{\ensuremath{\left(\frac{\Delta d}{\delta\epsilon}\log(L/\epsilon)\cdot\log(\delta \rho/\epsilon)\right)}}$.

\end{proof}

\subsection{Implementation of the oracles: proof of Lemma~\ref{lem:inner-product-oracle}}\label[sec]{sec:oracle}

In this section, we show how to convert \eqref{eq:expect_ug}  into a deterministic guarantee. 
\begin{lem}
Fix a unit vector $\hat{g}\in \R^d$ and let $z\in\R^{d}$ be a random vector satisfying $\mathbb{E}\langle\nabla f(z),\hat{g}\rangle\leq\frac{\epsilon}{3}$. 
Let  $z_1,\ldots, z_k$ be i.i.d realizations of $z$ with $k=\left\lceil \frac{36L}{\epsilon}\right\rceil \cdot \left\lceil\frac{\log(1/\gamma)}{\log(4)}\right\rceil$. Then with probability at least $1-\gamma$, one of the samples $z_i$ satisfies 
$\langle\nabla f(z_i),\hat{g}\rangle\leq\frac{\epsilon}{2}$.
%using
%%\[
%	\left\lceil \frac{36L}{\epsilon}\right\rceil \cdot \left\lceil\frac{\log(1/\gamma)}{\log(4)}\right\rceil \text{ gradient %evaluations
%5	of $f$.
%}	
\end{lem}
\begin{proof}
Define the random variable
$Y\defeq\langle\nabla f(z),\hat{g}\rangle$, and use $p\defeq\Pr[Y\leq\frac{\epsilon}{2}]$.
We note that 
\[
\E[Y]=p\cdot\E[Y\mid Y\leq\frac{\epsilon}{2}]+(1-p)\cdot\E[Y\mid Y>\frac{\epsilon}{2}].
\]
Rearranging the terms and using $\E[Y] \leq \epsilon/3$ gives
\[
p\cdot \left( \E[Y\mid Y>\frac{\epsilon}{2}]-\E[Y\mid Y\leq\frac{\epsilon}{2}]\right)\geq\frac{\epsilon}{6}.
\]
Finally, taking into account that $f$ is $L$-Lipschitz, we deduce  $|Y|\leq L$, which further implies $p\geq\frac{\epsilon}{12L}$. The results follows immediately. %Thus, we can find $\bar{z}$, with probability $1-\gamma$, using 
%$\left\lceil \frac{36L}{\epsilon}\right\rceil \cdot \left\lceil\frac{\log(1/\gamma)}{\log(4)}\right\rceil$ samples.
\end{proof}

\begin{lem}
Let $f\colon\R^{d}\to\RR$ be an $L$-Lipschitz continuous and $\rho$-weakly convex function. Fix a point $x$ and a unit vector $\hat g\in \RR^d$ such that $f$ is differentiable almost everywhere on the line segment $[x,y]$, where 
$y\defeq x-\delta\hat g$. Suppose that a random vector $z$ sampled uniformly from $[x,y]$ satisfies $\mathbb{E}_{z}\langle\nabla f(z),\hat{g}\rangle\leq\frac{\epsilon}{3}$.
Then, \cref{alg:binary_search} finds $\bar{z}\in \R^d$ such that $\langle\nabla f(\bar{z}),\hat{g}\rangle\leq\frac{\epsilon}{2}$
using $3\log(12\delta \rho/\eps)$ function evaluations of $f$.
\end{lem}

\begin{algorithm}[!ht]
    \DontPrintSemicolon
    \caption{Binary Search for $\bar z$}\label[alg]{alg:binary_search}
    \KwIn{Line Segment $[x,y=x-\delta \hat g]$}
    
    Let $[a,b] = [0,1]$

    \While{$b-a > \frac{\eps}{6\delta \rho}$}{

    \uIf {$f(x-a\delta \hat g) - f(x-\frac{a+b}{2}\delta \hat g) \leq   f(x-\frac{a+b}{2}\delta \hat g) - f(x-b\delta \hat g)$}{
        Let $[a,b] \leftarrow [a,\frac{a+b}{2}]$
    }
    \Else{
        Let $[a,b] \leftarrow [\frac{a+b}{2},b]$
    }

    }
    
    \textbf{Return} $x-a\delta \hat g$\; 
\end{algorithm}
\begin{proof}
Define the new function $h:[0,1]\rightarrow\R$ by 
$
h(t)=\langle\nabla f(x+t(y-x)),\hat{g}\rangle.
$
Clearly, we have 
\[
\frac{\epsilon}{3}\geq \E[h(t)]=\frac{1}{2}\underbrace{\mathbb{E}[h(t)\mid t\leq0.5]}_{P_\leq}+\frac{1}{2}\underbrace{\mathbb{E}[h(t)\mid t>0.5]}_{P_{>}}.
\]
Therefore $P_{\leq}$ or $P_>$ is at most $\epsilon/3$. The fundamental theorem of calculus directly implies
$P_{\leq}=\frac{f(x)-f(x-\frac{\delta}{2} \hat g)}{2\delta}$ and $P_>=\frac{f(x-\frac{\delta}{2} \hat g)-f(y)}{2\delta}$. Therefore with three function evaluations we may determine one of the two alternatives. Repeating this procedure
$\log(12\delta \rho/\eps)$ times, each times shrinking the interval by half, we can find an interval 
$[a,b]\subset [0,1]$ such that $b-a\leq\frac{\eps}{6\delta \rho}$ and $\mathbb{E}_{t\in[a,b]}h(t)\leq\frac{\epsilon}{3}$.
Note that for any $\bar t\in [a,b]$, we have  $h(\bar t)= \mathbb{E}h(t)+(h(\bar t)-\mathbb{E}h(t))$, while weak convexity implies 
\begin{align*}
h(\bar t)-\mathbb{E}h(t)&= \frac{1}{\delta}\EE_{t\in [a,b]}  \langle\nabla f(x+\bar t(y-x))-\nabla f(x+t(y-x)),x-y\rangle\\
&\leq \EE_{t\in [a,b]}\frac{\bar t- t}{\delta}\rho\|y-x\|^2\leq \frac{\epsilon}{6}.
\end{align*}
We thus conclude $h(\bar t)\leq \frac{\epsilon}{3}+ \frac{\epsilon}{6}=\frac{\epsilon}{2}$ as claimed.
\end{proof}

\newpage

\ifdefined\iscolt 
\else 
\bibliographystyle{plain} %do we want alpha inits? 
\fi 
\bibliography{bibliography}
\end{document}